\documentclass[a4paper,12pt]{article}

\usepackage{enumerate}
\usepackage{caption}

\usepackage{amsmath}
\usepackage{amsthm}
\usepackage{mathrsfs}
\usepackage{epsfig}
\usepackage{verbatim}
\usepackage[T1]{fontenc}
\usepackage{amsfonts}
\usepackage{float}
\usepackage{caption}
\usepackage{graphicx}
\usepackage{color}
\usepackage{fullpage}
\usepackage[utf8]{inputenc}

\usepackage{url}

\newtheorem{theorem}{Theorem}[section]
\newtheorem{proposition}[theorem]{Proposition}
\newtheorem{lemma}[theorem]{Lemma}
\newtheorem{corollary}[theorem]{Corollary}
\newtheorem{definition}[theorem]{Definition}

\newtheorem{example}[theorem]{Example}
\newtheorem{remark}[theorem]{Remark}

\newtheorem{problemx}{Problem}

\begin{document}



\title{Nonnegative measures belonging to $H^{-1}(\mathbb{R}^2).$ \\
\small{}}
\author{{Grzegorz Jamr\'oz} \\ \\
{ \small \it Institute of Mathematics, Polish Academy of Sciences, Śniadeckich 8, 00-656 Warszawa} \\
{\it \small e-mail: jamroz@impan.pl}}
\maketitle


\begin{abstract}
Radon measures belonging to the negative Sobolev space $H^{-1}(\mathbb{R}^2)$ are important from the point of view of fluid mechanics as they model vorticity of vortex-sheet solutions of incompressible Euler equations. In this note we discuss regularity conditions sufficient for nonnegative Radon measures supported on a line to be in $H^{-1}(\mathbb{R}^2)$. 
Applying the obtained results, we derive consequences for measures on $\mathbb{R}^2$ with arbitrary support and 
prove elementarily, among other things, that measures belonging to $H^{-1}(\mathbb{R}^2)$ may be supported on a set of Hausdorff dimension $0$. We comment on possible numerical applications.
\end{abstract}



\section{Introduction}
Let $\mathcal{M}_+(\mathbb{R}^2)$ denote the space of nonnegative bounded Radon measures on $\mathbb{R}^2$ (see \cite{EvaGar})  and let $H^{-1}(\mathbb{R}^2)$ be the space of all tempered distributions $f$ on $\mathbb{R}^2$ such that $$\int_{\mathbb{R}^2} (1+|y|^2)^{-1} |\hat{f}(y)|^2 dy < \infty.$$  
Alternatively, $H^{-1}(\mathbb{R}^2)$ can be viewed as the space of all continuous functionals on the Sobolev space $W^{1,2}(\mathbb{R}^2)$ (see e.g. \cite{AdaFou}). The following basic problem can be posed:

\begin{problemx}
\label{Prob1}
Characterize the space $\mathcal{M}_+(\mathbb{R}^2)  \cap H^{-1}(\mathbb{R}^2)$.
\end{problemx}

Our motivation to study this problem originates in fluid mechanics. Namely, let $u : \mathbb{R}^2 \to \mathbb{R}^2$ be the velocity field of a fluid in two-dimensional space and let $$\omega = {\rm curl} (u) := \partial_{x_1}u_2 - \partial_{x_2}u_1$$
be its vorticity field. Then $\omega \in \mathcal{M}_+(\mathbb{R}^2)  \cap H^{-1}(\mathbb{R}^2)$ for compactly supported $\omega$ means that 
\begin{itemize}
\item vorticity of the flow is everywhere nonnegative (condition $\omega \in \mathcal{M}_+(\mathbb{R}^2)$),
\item kinetic energy of the fluid is locally finite, i.e. $\int_{\Omega} u^2(x)dx< \infty$ for every bounded $\Omega \subset \mathbb{R}^2$ (condition $\omega \in H^{-1}(\mathbb{R}^2)$).
\end{itemize}

The latter condition follows from the fact that the Biot-Savart operator mapping $\omega$ to $u$ by the convolution formula
$$u = K*\omega$$ for $K(x) = \frac {x^{\perp}}{2\pi |x|^2}$ is bounded from $H^{-1}$ to $L^2_{loc}$, see below.
\\ 

 Solutions of the incompressible Euler equations,
\begin{eqnarray*}
\partial_t u + u \nabla u + \nabla p &=& 0, \\
{\rm div} (u)&=&0.
\end{eqnarray*}
with vorticity belonging to $\mathcal{M}_+(\mathbb{R}^2)$ were defined and studied in \cite{DiPMaj}. In \cite{Del} Delort proved a basic existence theorem, which states that 
for initial data $u(t=0,x)$ such that $\omega(0,x) := curl(u(0,x))$ is a bounded nonnegative Radon measure belonging to  $H^{-1}(\mathbb{R}^2)$ there exists a global solution $u(t,x)$ of the Euler equations such that $\omega(t,x) := curl(u(t,x))$ is a bounded nonnegative Radon measure belonging to $H^{-1}(\mathbb{R}^2)$ for every $t>0$. Uniqueness of such solutions is still an outstanding open problem. To approach it, it seems reasonable to study Problem \ref{Prob1}, see also the introduction in \cite{CieSzu} for a more comprehensive physical background and motivations. 

In the case of compactly supported measures Problem \ref{Prob1} can be solved as follows.
Define the positive logarithmic energy of a measure $\omega \in \mathcal{M}_+(\mathbb{R}^2)$ by
\begin{equation}
\mathcal{H}^+(\omega) := \int_{\mathbb{R}^2} \int_{\mathbb{R}^2} \log^+ \frac {1}{|x-y|} \omega(dx)\omega(dy),
\label{H+}
\end{equation}
where $\log^+(x) = \max(\log(x),0)$. In \cite{S}, which builds upon previous ideas of Delort \cite{Del} the following crucial characterization was demonstrated. 
\begin{lemma}[Lemma 3.1 in \cite{S}]
\label{LemSchochet}
Let $\omega$ be a nonnegative measure of finite mass and compact support, and let $u=K*\omega$ be the velocity corresponding to the vorticity $\omega$. Then the following are equivalent:
\begin{enumerate}
\item $\omega$ is in $H^{-1}$.
\item $u$ is in $L^2_{loc}$.
\item $\mathcal{H}^+(\omega)<\infty$.
\end{enumerate}
\end{lemma}

As a simple corollary, we obtain that measures belonging to $H^{-1}$ have no discrete part. Indeed, $\mathcal{H} ^+ (\delta _x) = + \infty$ for every $x \in \mathbb{R}^2$, where $\delta_x$ is the Dirac mass in $x$. For general measures, however, Formula \eqref{H+}  is not very convenient to use and we would like to have more 'tangible' local conditions characterizing measures belonging to $H^{-1}$.
\\ \, 

The study of Problem A in relation to spirals of vorticity was initiated in \cite{CieSzu}, where the authors proved that the so-called Prandtl and Kaden spirals belong locally to $H^{-1}(\mathbb{R}^2)$. The crucial tool in \cite{CieSzu} was the following theorem.

\begin{theorem}[Theorem 1.1 from \cite{CieSzu}]
\label{Th_CieSzu}
Let $\mu$ be a positive Radon measure supported in a ball $B(0,R_0) \subset \mathbb{R}^2$. Assume that there exists a positive constant $c_1$ such that for any $r \le R_0$ 
$$\mu(B(0,r)) = c_1r^{\alpha}, \mbox{where } \alpha>0.$$
Then $\mu \in H^{-1}(\mathbb{R}^2)$.
\end{theorem}

In this note, motivated by studies in \cite{CieSzu}, we go beyond Theorem \ref{Th_CieSzu}.
We investigate, namely, singular continuous measures belonging to $\mathcal{M}_+(\mathbb{R}^2) \cap H^{-1}(\mathbb{R}^2)$  and derive, using formula \eqref{H+}, simple analytical and geometric conditions characterizing such measures. We begin with measures supported on a line $\{(x_1,0): x_1 \in \mathbb{R}\}$ and then generalize the results to measures with more general support. In particular, we recover Theorem \ref{Th_CieSzu} as a special case. Let us note that our methods are based on transformation of formula \eqref{H+}, which, in contrast to t-energy methods (see \cite{Mat}) used in \cite{CieSzu} allow us to extract more detailed information on measures.
\\

Measure supported on a line can be written in the form
\begin{equation*}
\omega = \eta(dx_1)\delta_0(dx_2),
\end{equation*}
where $x=(x_1,x_2) \in \mathbb{R}^2$ and $\eta$ is a compactly supported nonnegative Radon measure on $\mathbb{R}$ with no discrete part. Measure $\omega$ can be equivalently represented as 
\begin{equation}
\label{eq_omegaF}
\omega = dF(x_1) \delta_0 (dx_2),
\end{equation}
where $F: \mathbb{R} \to [0,\infty)$ is the continuous, nondecreasing cumulative distribution function of $\eta$, given by
\begin{equation}
\label{eq_defF}
F(x) := \eta((-\infty,x]).
\end{equation}
If $\eta$ is absolutely continuous with respect to the one-dimensional Lebesgue measure or, equivalently, $F \in W^{1,1}_{loc} (\mathbb{R})$, then we can represent $\omega$ as 
\begin{equation}
\label{eq_omegaf}
\omega = f(x_1)dx_1 \delta_0 (dx_2),
\end{equation}
where $f:= F'$ is a nonnegative compactly supported function belonging to $L^1(\mathbb{R})$.
In the following, we study, under which conditions on $F$ and $f$ does $\omega$ belong to $H^{-1}(\mathbb{R}^2)$. We consider the following cases:
\begin{itemize}
\item $f \in L^1$,
\item $f \in L^{\infty}$ or equivalently $F$ -- Lipschitz continuous,
\item $f \in L^p$ for $1<p<\infty$,
\item $f \in L (\log L)^{\gamma}$, where $L (\log L)^{\gamma}$ is the Calder\'on-Zygmund class, see Section \ref{Sec3}.
\item $F$ -- continuous,
\item $F$ -- H\"older continuous with exponent $\alpha \in (0,1)$.
\end{itemize}
We prove that any of the conditions $f \in L^{\infty}$, $f \in L^p$, $F$ - H\"older continuous, $F$-Lipschitz continuous is sufficient (Section \ref{Sec2}). On the other hand, we show that conditions $f \in L^1$, $f \in L (\log L)^{\gamma}$ for $\gamma < 1\slash 2$ or $F$ being absolutely continuous are not sufficient (Section \ref{Sec3}). Finally (Section \ref{Sec4}) we apply these results to more general nonnegative measures $\omega$ and discuss the Hausdorff dimension of support of $\omega$. We comment also on possible numerical applications.
\section{Classes of measures belonging to $H^{-1}$}
\label{Sec2}
For measures $\omega$ of the form \eqref{eq_omegaF} formula \eqref{H+} reduces to
\begin{equation}
\mathcal{H}^+(\omega) = \mathcal{H}^+(dF) := \int_{\mathbb{R}} \int_{\mathbb{R}} \log^+ \frac {1}{|x-y|} dF(x)dF(y),
\label{H+F}
\end{equation}
where integrals are understood in the Lebesgue-Stieltjes sense (i.e. $dF \equiv \eta$ is the Lebesgue-Stieltjes measure generated by equality \eqref{eq_defF}, see \cite{CarBru}). Similarly, for measures $\omega$ of the form \eqref{eq_omegaf}, we obtain
\begin{equation}
\mathcal{H}^+(\omega)= \mathcal{H}^+(f) := \int_{\mathbb{R}} \int_{\mathbb{R}} \log^+ \frac {1}{|x-y|} f(x)f(y) dx dy.
\label{H+f}
\end{equation}
So prepared, we are ready to study particular cases of Problem \ref{Prob1}. By Lemma \ref{LemSchochet}, it suffices to determine whether  $\mathcal{H}^+(dF)$ or $\mathcal{H}^+(f)$ are finite, using formulas \eqref{H+F} and \eqref{H+f}, respectively. We begin with the simple cases of $f \in L^{\infty}$ and $f \in L^p$, $p>1$.

\begin{proposition}
If $f$ is bounded and compactly supported then $\mathcal{H}^+(f) < \infty$.
\end{proposition}
\begin{proof}
$$ \mathcal{H}^+(f) \le \|f\|_{L^{\infty}}^2 \int_{{\rm supp}(f)}\int_{{\rm supp}(f)} \log^+ \frac {1}{|x-y|} dxdy < \infty, $$
where ${\rm supp}(f)$ denotes the support of function $f$.
\end{proof}
\begin{corollary}
For $F$ Lipschitz continuous $\mathcal{H}^+(dF) < \infty$.
\end{corollary}

\begin{proposition}
\label{Prop_lp}
If $f \in L^p$, $1 < p \le \infty$ and $f$ is compactly supported then $\mathcal{H}^+(f) < \infty$.
\end{proposition}
\begin{proof}
Let $f \in L^p$ have a compact support such that ${\rm supp} (f) \subset B(0,R)$, where $B(0,R)$ is the closed ball centered at $0$ and with radius $R$. Then, setting $q$ such that $\frac 1 p + \frac 1 q  = 1$ and using the H\"older  and Young inequalities we obtain
\begin{eqnarray*}
\int_{\mathbb{R}}\int_{\mathbb{R}} \log^+ \frac 1 {|x-y|} f(x)f(y)dxdy &=& \int_{{B(0,R+1)}}\int_{B(0,R+1)} \log^+ \frac 1 {|x-y|} f(x)f(y)dxdy\\
&\le& \left\| \int_{B(0,R+1)}  \log^+ \frac 1 {|\cdot -y|} f(y)dy \right\|_q \|f\|_p \\
&\le&  \left\| \int_{B(0,R+1)}  \log^+ \frac 1 {|\cdot -y|} f(y)dy \right\|_{\infty} [2(R+1)]^{\frac 1 q} \|f\|_p \\
&\le&  \left\|\log^+ \frac {1}{|\cdot|} \bold{1}_{B(0,R+1)} (\cdot)\right\|_q  [2(R+1)]^{\frac 1 q} \|f\|_p^2 < +\infty.
\end{eqnarray*}
\end{proof}




Next, we consider the more demanding case of $F$ being H\"older continuous. Recall that $F \in C^{0,\alpha}(\mathbb{R})$, $0<\alpha\le 1$, if there exists a constant $K>0$ such that $|F(x+y) - F(x)| \le K |y|^{\alpha}$ for every $x,y \in \mathbb{R}$.

\begin{proposition}
\label{Prop_Calpha}
If $F \in C^{0,\alpha}$, $0 < \alpha \le 1$ then $\mathcal{H}^+(dF) < \infty$.
\end{proposition}

Proposition \ref{Prop_Calpha} is a consequence of the following lemma.

\begin{lemma}
\label{Lem_Calpha}
Suppose a bounded continuous nondecreasing $F:\mathbb{R} \to [0,\infty)$ satisfies:
\begin{enumerate}[i)]
\item $(F(x+\varepsilon) - F(x)) \log \varepsilon \to 0$ as $\varepsilon \to 0$ uniformly in $x$,
\item $(F(x-\varepsilon) - F(x)) \log \varepsilon \to 0$ as $\varepsilon \to 0$ uniformly in $x$,
\item $\int_0^1 \frac {F(x+y) - F(x)}{y} dy \le C$ uniformly in $x$,
\item $\int_0^1 \frac {F(x) - F(x-y)}{y} dy \le C$ uniformly in $x$.
\end{enumerate}
Then 
$$\mathcal{H^+}(dF) = \int_{\mathbb{R}} \left( \int_{0}^{1} \frac 1 {y}(F(x+y)-F(x-y)) dy \right) dF(x) $$
and in particular, $\mathcal{H^+}(dF) < +\infty$.
\end{lemma}
\begin{proof}
Using the properties of Lebesgue-Stieltjes integrals (see \cite{CarBru}) we obtain:
\begin{eqnarray*}
\mathcal{H}^+(dF) &=& \int_{\mathbb{R}} \int_{\mathbb{R}} \log^+ \frac {1}{|x-y|} dF(x)dF(y)\\
&=& \int_{\mathbb{R}} \left[\int_{x-1}^{x+1} \log \frac 1 {|x-y|} dF(y) \right] dF(x)\\ 
&=& \int_{\mathbb{R}} \left[\int_{-1}^{1} \log \frac 1 {|y|} dF(x+y) \right] dF(x)\\ 
&=& \int_{\mathbb{R}} \left[\int_{0}^{1} \log \left(\frac 1 {y}\right) d(F(x+y)-F(x-y)) \right] dF(x)\\ 
&=& \int_{\mathbb{R}} \int_{0}^{1} \log \left(\frac 1 {y}\right) d(F(x+y)-F(x)) dF(x) \\  && + \int_{\mathbb{R}} \int_{0}^{1} \log \left(\frac 1 {y}\right) d(F(x)-F(x-y)) dF(x)\\ 
&=& \int_{\mathbb{R}} \lim_{\varepsilon \to 0}\left[\int_{\varepsilon}^{1} \log \left(\frac 1 {y}\right) d(F(x+y)-F(x)) \right] dF(x)\\ 
&& + \int_{\mathbb{R}} \lim_{\varepsilon \to 0}\left[\int_{\varepsilon}^{1} \log \left(\frac 1 {y}\right) d(F(x)-F(x-y)) \right] dF(x)\\ 
&=& \int_{\mathbb{R}} \lim_{\varepsilon \to 0}\left[\left[ \log\left(\frac 1 y \right)(F(x+y) - F(x))\right]_\varepsilon^1  + \int_{\varepsilon}^{1} \frac 1 {y}(F(x+y)-F(x)) dy \right] dF(x)\\ 
&&  + \int_{\mathbb{R}} \lim_{\varepsilon \to 0}\left[\left[ \log\left(\frac 1 y \right)(F(x) - F(x-y))\right]_\varepsilon^1  + \int_{\varepsilon}^{1} \frac 1 {y}(F(x)-F(x-y)) dy \right] dF(x)\\ 
&=& \int_{\mathbb{R}} \left( \int_{0}^{1} \frac 1 {y}(F(x+y) - F(x) + F(x) -F(x-y)) dy \right) dF(x)  \le 2C \int_{\mathbb{R}} dF(x),
\end{eqnarray*}
where in the last equality we used the Lebesgue dominated convergence theorem and the fact that measure $dF$ is bounded.
\end{proof}

\begin{proof}[Proof of Proposition \ref{Prop_Calpha}]
For $F \in C^{0,\alpha}$, where $0 < \alpha \le 1$,  we obtain
\begin{equation*}
|F(x \pm  \varepsilon) - F(x)| \log(\varepsilon) \le K\varepsilon^{\alpha} \log(\varepsilon) \to 0
\end{equation*}

as $\varepsilon \to 0$ and
\begin{equation*}
\int_0^1 \frac {|F(x \pm y) - F(x)|}{y} dy \le K\int_0^1 y^{\alpha -1} dy = K \slash \alpha.
\end{equation*}
Using Lemma \ref{Lem_Calpha} we conclude.
\end{proof}

\begin{remark}
\label{Rem_UniformH}
\rm
Proofs of Lemma \ref{Lem_Calpha} and Proposition \ref{Prop_Calpha} show that if $F$ satisfies \\$|F(x+y)-F(x)| \le K |y|^{\alpha}$ then $$\mathcal{H}^+(dF) \le 2(K \slash \alpha) \omega(\mathbb{R}^2).$$
\end{remark}

\begin{remark}
\rm
\label{Rem_logbeta}
Conditions i)-iv) from Lemma \ref{Lem_Calpha} encompass a larger class of functions than functions which are H\"older continuous. For instance, it suffices to assume that $|F(x+y) - F(x)| \le 1 \slash |\log(|y|)|^{\beta}$ for $|y|\le \varepsilon$, $x \in \mathbb{R}$ and fixed $\beta > 1$ and $\varepsilon>0$.
\end{remark}

\begin{remark}
\rm
Due to embedding $W^{1,p}(\mathbb{R}) \hookrightarrow C^{0, 1 - 1\slash p} (\mathbb{R})$ for $p>1$ (see e.g. \cite{AdaFou}), using Proposition \ref{Prop_Calpha} we recover the result from Proposition \ref{Prop_lp}.
\end{remark}

\begin{remark}
\label{RemH-12}
\rm
Results of this section allow us to obtain embeddings of various spaces into the fractional Sobolev space $H^{1\slash 2}$ (see \cite{Tar}) 
as follows.
Distributions belonging to $H^{-1}(\mathbb{R}^2)$, which are supported on the line $\{(x_1,0): x_1 \in \mathbb{R} \}$ may be identified with the space of $H^{- 1\slash 2}(\mathbb{R})$ due to the fact that the trace operator $T: W^{1,2}(\mathbb{R}^2) \to H^{1\slash 2}(\mathbb{R})$ is bounded and has a  bounded right inverse, see \cite[Section 16]{Tar}. 
Hence, if $\omega \in H^{-1}(\mathbb{R}^2)$ is of the form \eqref{eq_omegaF} then $dF$ belongs to $H^{-1\slash 2}(\mathbb{R})$ and consequently $F$ belongs locally to $H^{1 \slash 2}$. Now, Proposition \ref{Prop_Calpha}, for instance, allows us to obtain a local embedding of nondecreasing functions belonging to $C^{0,\alpha}$, $0 < \alpha <1$, into $H^{1 \slash 2}$. 
\end{remark}

It is not possible to extend the results of this section to arbitrary absolutely continuous $F$. In the next section we show counterexamples. 

\section{Counterexamples}
\label{Sec3}
We begin by describing a class of functions, which we will use for construction of counterexmaples for $f \in L^1$ and $f \in L(\log L)^{\gamma}$. Let, namely,
\begin{equation*}
f(x) = \sum_{n=1}^{\infty} h_n \bold{1}_{[a_n,a_n+d_n]}(x),
\end{equation*}
where for every $n=1,2, \dots$ we have $a_n \in \mathbb{R}$, $h_n>1$, $0 < d_n \le 1$ and $a_n + d_n \le a_{n+1}$. Observe that 

\begin{equation*}
\mathcal{H}^+(h \bold{1}_{[a,a+d]}) \ge  \int_{a}^{a+d} \int_{a}^{a+d} \log^+ \frac {1}{|x-y|} h^2 dx dy \ge h^2d^2 \log(1/d)
\end{equation*}
and hence 
\begin{equation}
\label{HH+}
\mathcal{H}^+(f) \ge \sum_{n=1}^{\infty} h_n^2 d_n^2 \log(1 \slash d_n).
\end{equation}
\begin{proposition}
\label{Prop_L1}
There exists a nonnegative compactly supported $f \in L^1$ such that $\mathcal{H}^+(f) = +\infty$.
\end{proposition}
\begin{proof}
Take $d_n = \exp(-2^{2n})$ and $h_n = 1 \slash (2^n d_n)$. Then on the one hand
$$\|f\|_{L^1} = \sum_{n=1}^{\infty}h_n d_n = 1.$$
On the other hand, however, by \eqref{HH+}
$$\mathcal{H}^+(f) \ge \sum_{n=1}^{\infty} 2^{-2n} \log(1/d_n) = +\infty.$$
\end{proof}

\begin{corollary}
There exists an absolutely continuous $F$ such that $\mathcal{H}^+(dF) = +\infty$.
\end{corollary}

Using the same construction we can generalize the result to the Calder\'on-Zygmund class $L (\log L)^{\gamma}$, for $\gamma < 1 \slash 2$. Recall that $f \in L (\log L)^{\gamma} (\mathbb{R})$ if 
$$\int_{\mathbb{R}} |f(x)| (\log(1+|f(x)|))^{\gamma} dx < \infty.$$
\begin{proposition}
For every $\gamma< 1 \slash 2$ there exists a nonnegative compactly supported $f \in L (\log L)^{\gamma}$ such that $\mathcal{H}^+(f) = +\infty$.
\end{proposition}
\begin{proof}
A direct calculation shows that function $f$ constructed in Proposition \ref{Prop_L1} belongs in fact to $L(\log L)^{\gamma}$ for every $\gamma < 1 \slash 2$.
\end{proof}

\section{Applications}
\label{Sec4}

To apply the results of the previous sections it is useful to generalize them to the two-dimensional setting. We begin by defining the radial cumulative distribution function of a measure $\omega \in \mathcal{M}_+(\mathbb{R}^2)$.

\begin{equation}
\label{eq_defG}
G(r) := \begin{cases} \omega(B(0,r)) &\mbox{ for } r>0, \\ 0 &\mbox{ otherwise,}\end{cases}
\end{equation}
where $B(0,r)$ is the closed ball centered at $0$ and with radius $r$.
\noindent Using $G(r)$ we estimate $\mathcal{H}^+(\omega)$ by $\mathcal{H}^+(dG)$ as follows.
\begin{lemma}
\label{Lem_Appl}
Let $\omega$ be a compactly supported nonnegative Radon measure on $\mathbb{R}^2$. 
Let $G$ be its radial cumulative distribution function defined by \eqref{eq_defG}. Then

\begin{enumerate}[i)]
\item  
for every Borel function $h: [0,\infty) \to [0,\infty)$
\begin{equation}
\label{Eq_omegaradial}
\int_{\mathbb{R}^2} h(|x|)\omega(dx) = \int_{[0,{\infty})} h(r) dG(r),
\end{equation}

\item $\mathcal{H}^+(\omega) \le \mathcal{H}^+(dG)$.
\end{enumerate}
\end{lemma}

\begin{remark}
\label{Rem1}
\rm
The reverse inequality in Lemma \ref{Lem_Appl}ii is false even up to a constant. For instance, both $\nu_1=\delta_{(1,0)}$ and $\nu_2$ -- a probability measure distributed uniformly on the circle $\{(x_1,x_2): x_1^2 + x_2^2 = 1\}$ have the same radial cumulative distribution function $$G(r) = \bold{1}_{[1,\infty)}(r).$$ Nevertheless, $\mathcal{H}^+(\nu_1) = \mathcal{H}^+(dG) = \infty$ yet $\mathcal{H}^+(\nu_2) < \infty$, see Remark \ref{Rem2}. 
\end{remark}
\begin{remark}
\label{Rem2}
\rm
Inequality in Lemma \ref{Lem_Appl}ii holds for $G$ centered at any $x_0 \in \mathbb{R}^2$, i.e.  $\mathcal{H}^+(\omega) \le \mathcal{H}^+(dG_{x_0})$ for
\begin{equation*}
G_{x_0}(r) := \begin{cases} \omega(B(x_0,r)) &\mbox{ for } r>0, \\ 0 &\mbox{ otherwise.}\end{cases}
\end{equation*}
The choice of $x_0$ is important in order to obtain a useful estimate. Taking, for instance, $x_0 = (1,0)$ we obtain for measure $\nu_2$ from Remark \ref{Rem1} that 
$$G_{x_0}(r) = \begin{cases} 0 &\mbox{ for } r<0, 
\\  (2\slash \pi) \arcsin(r\slash 2)   &\mbox{ for } 0 \le r \le {2},
\\  1   &\mbox{ for } 2 \le r,
\end{cases}$$
which is H\"older continuous with exponent $1 \slash 2$. Thus, $\mathcal{H}^+(\nu_2) \le \mathcal{H}^+(dG_{x_0}) < \infty$. On the other hand, the choice $x_0 = (0,0)$ leads to $\mathcal{H}^+(\nu_2) \le \mathcal{H}^+(dG_{x_0}) = \mathcal{H}^+(\delta_1) = \infty$, which does not allow us to conclude about finiteness of $\mathcal{H}^+(\nu_2)$.
\end{remark}

\begin{proof}[Proof of Lemma \ref{Lem_Appl}]

i) By definition of $G$, equality \eqref{Eq_omegaradial} holds for $h(r) = \bold{1}_{[r_1,r_2]}(r)$ with any $0 \le r_1 < r_2 \le \infty$. Standard approximation arguments for Radon measures and the Lebesgue monotone convergence theorem allow us to prove the case of general $h$.

ii) We observe that $\log^+ \frac {1}{|x-y|} \le \log^+ \frac {1}{||x|-|y||}$, use repeatedly representation from i) as well as the Fubini theorem and calculate:
\begin{eqnarray*}
 \mathcal{H}^+(\omega) &=& \int_{\mathbb{R}^2} \int_{\mathbb{R}^2} \log^+ \frac {1}{|x-y|} \omega(dx)\omega(dy)\\
&\le &\int_{\mathbb{R}^2} \int_{\mathbb{R}^2} \log^+ \frac {1}{||x|-|y||} \omega(dx)\omega(dy)\\
&=& \int_{\mathbb{R}^2} \left[\int_{[0,{\infty})} \log^+ \frac {1}{|r_x-|y||} dG(r_x) \right]\omega(dy)\\
&=& \int_{[0,{\infty})} \int_{[0,{\infty})} \log^+ \frac {1}{|r_x-r_y|} dG(r_x)dG(r_y)\\
&=& \int_{\mathbb{R}} \int_{\mathbb{R}} \log^+ \frac {1}{|r_x-r_y|} dG(r_x)dG(r_y) = \mathcal{H}^+(dG).\\
\end{eqnarray*}

\end{proof}

\begin{corollary}
\label{CorCie}
Fix $\alpha > 0$ and let $\omega$ be a Radon measure such that $\omega(B(0,r))=G(r)$ for 
\begin{equation}
\label{GCieSzu}
G(r)=
\begin{cases} cr^{\alpha} &\mbox{ for }  0 \le r \le R,\\
cR^{\alpha} &\mbox{ for r >R,} \\
0 &\mbox { otherwise.}
\end{cases}
\end{equation}
Then $\omega \in H^{-1}(\mathbb{R}^2)$. Thus,  we recover Theorem \ref{Th_CieSzu}. 
\end{corollary}
\begin{proof}
$\mathcal{H}^+(dG) < + \infty$, which follows by the fact that $G'(r) = c\alpha r^{\alpha -1} \bold{1}_{[0,R]}(r)$ belongs to $L^p$ for some $p>1$. Using Proposition \ref{Prop_lp} and Lemmas \ref{LemSchochet}, \ref{Lem_Appl} we conclude. Alternatively, we can use Proposition \ref{Prop_Calpha}, observing that $G(r) \in C^{0,\alpha}$.
\end{proof}


Next, let us investigate the Hausdorff dimension of the support of measures belonging to $H^{-1}(\mathbb{R}^2)$. As we will use Cantor sets and Cantor functions, we recall the definitions and basic properties of them.

\begin{definition}
\label{DefCantor}
\begin{enumerate}[i)]
\item The standard Cantor set is the set $C \subset [0,1]$ constructed inductively as follows. 
\begin{itemize}
\item $Z_0 = [0,1]$. 
\item $Z_1$ is obtained from $Z_0$ by removing the middle third of the interval, i.e. $Z_1 = [0,1\slash 3] \cup [2\slash 3,1]$. 
\item $Z_2$ is obtained from $Z_1$ by removing the middle third of every remaining interval in $Z_1$, i.e. $Z_2 = [0,1\slash 9] \cup [2\slash 9,1 \slash 3] \cup  [2 \slash 3,7\slash 9] \cup [8\slash 9,1]$.
\item $Z_n$ is, in general, obtained by removing the middle third of every remaining interval in $Z_{n-1}$.
\end{itemize}
Finally, $C:= \bigcap_{n=1}^{\infty} Z_n.$
\item The standard Cantor function $\Gamma:[0,1] \to [0,1]$ can be constructed inductively as follows. 
\begin{itemize}
\item $\gamma_0  (x) = x$
\item $\gamma_n(x) = \begin{cases} 1\slash 2\gamma_{n-1}(3x) &\mbox{ for } 0 \le x < 1\slash 3, \\
1 \slash 2 &\mbox { for } 1\slash 3 \le x \le 2 \slash 3,\\
1\slash 2 +  1\slash 2\gamma_{n-1}(3x - 2) &\mbox{ for } 2\slash 3 < x \le 1.
\end{cases}$
\end{itemize}
We define $\Gamma := \lim_{n\to \infty} \gamma_n$, where the convergence is uniform on $[0,1]$. If we prolong $\Gamma$ by $0$ for $x \le 0$ and $1$ for $x \ge 1$ then we obtain a nondecreasing continuous function mapping $\mathbb{R}$ onto $[0,1]$.
\end{enumerate}
\end{definition}
\noindent Let us summarize the basic properties of the standard Cantor set and Cantor function useful later on. For the proofs, we refer the reader to the survey paper \cite{DovMarRia}.
\begin{proposition}
\begin{enumerate}[i)]
\item The standard Cantor set is closed.
\item The dimension of the standard Cantor set equals $\log(2)\slash \log(3)$.
\item The standard Cantor function is H\"older continuous with exponent $\log(2) \slash \log(3)$.
\item Measure $d\Gamma$ is supported on $C$.
\end{enumerate}
\end{proposition}

\begin{example}
\label{Cor_Cantor1}
Let $\omega$ satisfy $$\omega(B(0,r)) = \Gamma(r),$$ where $\Gamma(r)$ is the standard Cantor function. 
 Then $\omega \in H^{-1}(\mathbb{R}^2)$. 
\end{example}
\begin{proof}
$\Gamma(r)$ is H\"older continuous with exponent $\alpha = \log (2) \slash \log (3)$. The assertion follows by Proposition \ref{Prop_Calpha} and Lemmas \ref{LemSchochet} and \ref{Lem_Appl}.
\end{proof}

Now, we are ready to construct examples of measures belonging to $H^{-1}(\mathbb{R}^2)$ supported on very small sets.
\begin{proposition}
\label{Prop_Cantor2}
A nonnegative Radon measure belonging to $H^{-1}(\mathbb{R}^2)$ may be supported on a set of arbitrary small positive Hausdorff dimension.
\end{proposition}
\begin{proof}
Consider a modified Cantor set $C_K$ obtained by removing in every step of the construction, described in Definition \ref{DefCantor}, the middle $(K-2)\slash K$ portion of every interval (note that for $K=3$ we obtain the standard Cantor set). Let $\Gamma_K(r)$ be the corresponding Cantor function, constructed similarly as in Definition \ref{DefCantor}, and consider the measure
$$\omega_K = d\Gamma_K (x_1) \delta_0(dx_2).$$
Then measure $\omega_K$ is supported on the closed set $C_K$ of dimension $\alpha = \log(2)\slash \log(K)$. Moreover, $\Gamma_K(r)$ is H\"older continuous with the same exponent $\alpha = \log(2)\slash \log(K)$, see e.g. \cite{GorKuk}, and hence $\omega_K \in H^{-1}(\mathbb{R}^2)$. 
\end{proof}
Adapting the above construction, we can prove that a measure belonging to $H^{-1}(\mathbb{R}^2)$ may be supported on a set of Hausdorff dimension $0$.
\begin{proposition}
There exists a nonnegative bounded Radon measure belonging to $H^{-1}(\mathbb{R}^2)$ which is supported on a bounded set of Hausdorff dimension $0$.
\end{proposition}
\begin{proof}[Sketch of the proof.]
We construct a general Cantor set $C_{\infty}$ by removing in step $n$ of the construction the central $1- 2c_n$ portion of every interval remaining from step $n-1$. 
We obtain
\begin{itemize}
\item $Z_{\infty}^0 = [0,1]$, 
\item $Z_{\infty}^1 = [0,c_1] \cup [1-c_1,1]$, 
\item $Z_{\infty}^2 = [0,c_1c_2] \cup [c_1 -c_1c_2,c_1] \cup [1-c_1, 1-c_1 + c_1c_2]  \cup [1-c_1c_2,1]$,
\item \dots
\end{itemize}
(note that $c_n\equiv 1 \slash 3$ would lead to the standard Cantor set).
Observe that the length of every of the $2^n$ intervals constituting $Z^n_{\infty}$ is equal 
$$d_n = {c_1c_2 \dots c_n}. $$ 
Fix $\beta>1$ and set $$d_n = e^{- 2^{n \slash \beta}}.$$ 
Then $c_n =  {d_{n}}\slash {d_{n-1}}$ is decreasing and tends to $0$ as $n \to \infty$.
Define $$C_\infty:= \bigcap_{n=0}^{\infty} Z^n_{\infty}.$$
Observe that  $Z_{\infty}^n$ is a union of $2^n$ intervals of length $d_n$ and hence $C_{\infty}$ can be covered by $2^n$ balls of diameter $d_n$ for $n=1,2,  \dots$. Since for every fixed $\varepsilon > 0$ we have $2^n (d_n)^{\varepsilon} \to 0$ as $n \to \infty$, we conclude that the Hausdorff dimension of $C_{\infty}$ is equal $0$.


Define
$$\omega_{\infty} := d\Gamma_{\infty} (x_1) \delta_0(dx_2),$$
where $\Gamma_{\infty}$ is the corresponding Cantor function constructed as in Definition \ref{DefCantor}. More precisely, let 
\begin{itemize}
\item $\gamma_{\infty}^0(x) = x$
\item $\gamma_{\infty}^n(x) = \begin{cases} 1\slash 2\gamma^{n-1}_{\infty}(x\slash c_n) &\mbox{ for } 0 \le x < c_n, \\
1 \slash 2 &\mbox { for } c_n \le x \le 1-c_n,\\
1\slash 2 +  1\slash 2\gamma^{n-1}_{\infty}((x - 1+ c_n) \slash c_n) &\mbox{ for } 1-c_n < x \le 1.
\end{cases}$
\end{itemize}
and define $\Gamma_\infty := \lim_{n\to \infty} \gamma_n$, prolonging it by $0$ for $x \le 0$ and $1$ for $x \ge 1$.
We claim that $$\Gamma_{\infty} (y) \le 1 \slash |\log(|y|)|^{\beta}$$ for $y \le \exp(- (\beta+1))$. Indeed,
\begin{itemize}
\item function $y \mapsto 1 \slash |\log(|y|)|^{\beta}$ is increasing on the interval $[0,1]$,
\item function $y \mapsto 1 \slash |\log(|y|)|^{\beta}$ is concave on the interval $[0,\exp(- (\beta+1))]$,
\item $\Gamma_{\infty} (d_n) =  2^{-n} = 1 \slash |\log(|d_n|)|^{\beta}$ for $n=0,1, \dots$,
\item the graph of $\Gamma_{\infty}$ restricted to $[d_{n+1},d_n]$ lies below the segment connecting points $(d_{n+1},\Gamma_{\infty}(d_{n+1}))$ and $(d_{n},\Gamma_{\infty}(d_{n}))$, i.e.
$$\Gamma_{\infty}(y) \le \Gamma_{\infty}(d_{n+1}) + \frac {y-d_{n+1}}{d_n - d_{n+1}}(\Gamma_{\infty } (d_n) - \Gamma_{\infty}(d_{n+1}))$$ for every $y \in [d_{n+1},d_n]$,
\item the segment connecting points $(d_{n+1},\Gamma_{\infty}(d_{n+1}))$ and $(d_{n},\Gamma_{\infty}(d_{n}))$ lies, for $n$ satisfying $d_n \le \exp(-(\beta+1))$, below the graph of $y \mapsto 1 \slash |\log(|y|)|^{\beta}$ due to concavity of the latter function.
\end{itemize}
Consequently, $\Gamma_{\infty} (y) \le 1 \slash |\log(|y|)|^{\beta}$ for $0 \le y \le \exp(- (\beta+1)).$ Self-similarity of $\Gamma_{\infty}$ allows us to conclude that  $$|\Gamma_{\infty} (x+ y) - \Gamma_{\infty}(x)| \le 1 \slash |\log(|y|)|^{\beta}$$ for $|y| \le \exp(- (\beta+1))$ and arbitrary $x \in \mathbb{R}$. 
Using Remark \ref{Rem_logbeta} and Lemma \ref{Lem_Calpha} we obtain  
$\mathcal{H}^+(d\Gamma_{\infty}) < + \infty$
and hence $\omega_{\infty} \in H^{-1}(\mathbb{R}^2)$.
\end{proof}
Finally, let us briefly comment on possible numerical applications of our results.
\begin{remark}
\rm 
From the point of view of proving the convergence of numerical schemes it is important to know that $\omega^n$, a sequence of approximations of a compactly supported measure $$\omega \in \mathcal{M}_+(\mathbb{R}^2) \cap H^{-1}(\mathbb{R}^2),$$ is such that $\mathcal{H}^+(\omega^n)$ remains bounded uniformly in $n$  (see e.g. \cite{S} or \cite{LopLowNusZhe}). Let, for instance, $\omega$ be the positive branch of the Kaden spiral (see \cite{CieSzu}) at some point in time. Then function $r \mapsto \omega(B(0,r))$ is H\"older continuous with exponent $\alpha = 1\slash 2$ (see \cite{CieSzu}) and hence belongs locally to $H^{-1}(\mathbb{R}^2)$. Let $\omega_n$ be a smooth approximation of $\omega$, e.g. a vortex blob approximation, see \cite{LopLowNusZhe}.  To prove that $\mathcal{H}^+(\omega^n)$ is bounded uniformly with respect to $n$ it suffices, by Remark \ref{Rem_UniformH}, to show that functions  $$r \mapsto \omega^n(B(0,r))$$ are uniformly H\"older continuous with constant $K$ and exponent $\alpha$ indepenent of $n$. Whether this is the case, depends on a particular form of vortex blob approximation. The goal is then to construct an approximation which satisfies the uniform H\"older condition. This, however, is relatively simple, since $r \mapsto \omega(B(0,r))$ is H\"older continuous.
\end{remark}
{\bf Acknowledgements.} I am grateful to Tomasz Cie\'slak from the Institute of Mathematics, Polish Academy of Sciences in Warsaw for reading the manuscript and valuable comments regarding it. I also acknowlegde his drawing my attention to numerical applications of the obtained results. Furthermore, I am grateful to Marcin Ma\l{}ogrosz from the Instutite of Applied Mathematics and Mechanics, University of Warsaw for a useful discussion concerning Remark \ref{RemH-12}. 


\end{document}